\documentclass{amsart}

\usepackage{amssymb}

\begin{document}

\title[Analysis of a class of variational multiscale methods]
{An analysis of a class of variational multiscale methods 
 based on subspace decomposition}

\author[Kornhuber]{Ralf Kornhuber}
\address{Institut f\"{u}r Mathematik, 
 Freie Universit\"{a}t Berlin, 
 14195 Berlin, Germany}
\email{kornhuber@math.fu-berlin.de}

\author[Peterseim]{Daniel Peterseim}
\address{Institut f\"{u}r Mathematik, 
 Universit\"{a}t Augsburg, 
 86135 Augsburg, Germany}
\email{daniel.peterseim@math.uni-augsburg.de}
 
\author[Yserentant]{Harry Yserentant}
\address{Institut f\"{u}r Mathematik, 
 Technische Universit\"{a}t Berlin, 
 10623 Berlin, Germany}
\email{yserentant@math.tu-berlin.de}

\thanks{This research was supported by Deutsche 
 Forschungsgemeinschaft through SFB 1114}


%
% 65N12: Stability and convergence of numerical methods
%
% 65N30: Finite elements, Rayleigh-Ritz and Galerkin 
%        methods, finite methods
%       
% 65N55: Multigrid methods; domain decomposition       
%

\begin{abstract}
Numerical homogenization tries to approximate the solutions 
of elliptic partial differential equations with strongly 
oscillating coefficients by functions from modified finite 
element spaces. We present a class of such methods that 
are closely related to the methods that have recently been 
proposed by M\r{a}lqvist and Peterseim 
[Math. Comp. 83, 2014, pp. 2583--2603]. Like these methods, 
the new methods do not make explicit or implicit use of a 
scale separation. Their comparatively simple analysis is 
based on the theory of additive Schwarz or subspace 
decomposition methods.
\end{abstract}

\maketitle

%%%%%%%%%%%%%%%%%%%%%%%%%%%%%%%%%%%%%%%%%%%%%%%%%%%%%%%%%%%%%%%%%%%%%%%%%%

\newtheorem{theorem}{Theorem}[section]
\newtheorem{lemma}[theorem]{Lemma}

\numberwithin{equation}{section}

\def  \I     {\mathcal{I}}
\def  \N     {\mathcal{N}}
\def  \S     {\mathcal{S}}
\def  \T     {\mathcal{T}}
\def  \V     {\mathcal{V}}
\def  \W     {\mathcal{W}}

\def  \dx    {\,\mathrm{d}x}

\def  \uS    {\widetilde{u}}

\def  \ker   {\mathrm{ker}\,}
\def  \span  {\mathrm{span}\,}

\def  \sp    {\mspace{1mu}}

%%%%%%%%%%%%%%%%%%%%%%%%%%%%%%%%%%%%%%%%%%%%%%%%%%%%%%%%%%%%%%%%%%%%%%%%%%

%Section 1
\section{Introduction}
\label{sec1}

Numerical homogenization aims at a modification of standard 
finite element discretizations that preserve the accuracy 
known from smooth coefficients functions to the case of highly 
oscillatory coefficient functions. A method of this kind, 
which utilizes no separation of scales at all and is founded 
on a comprehensive convergence theory, has recently been 
proposed by M\r{a}lqvist and one of the authors 
\cite{Malqvist-Peterseim}. The central idea of this paper, 
which can be considered as a late descendant of the work 
of Babu\v{s}ka and Osborn \cite{Babuska-Osborn}, is to 
assign to the vertices of the finite elements new basis 
functions that span the modified finite element space 
and reflect the multiscale structure of the problem under 
consideration. In the basic version of the method, the new 
basis functions are the standard piecewise linear hat 
functions minus their orthogonal projection onto a space 
of rapidly oscillating functions, orthogonal with respect 
to the symmetric, coercive bilinear form associated with 
the boundary value problem. They possess a global support 
but decay exponentially from one shell of elements 
surrounding the assigned vertex to the next. It is 
therefore possible to replace them by local counterparts 
without sacrificing the accuracy. The support of these 
localized basis functions consists of a fixed number of 
shells of elements surrounding the associated node. The 
number of these shells increases logarithmically with 
increasing accuracy, that is, decreasing gridsize. 

Here we present a considerably simplified analysis of a class 
of closely related methods that is based on the theory of 
iterative methods, more precisely of additive Schwarz or 
subspace decomposition methods \cite{Xu}, \cite{Yserentant}, 
and essentially utilizes a reinterpretation \cite{Peterseim}
of the method proposed in \cite{Malqvist-Peterseim} in terms 
of variational multiscale methods. Arguments used in the 
present paper are inspired by the work of Demko~\cite{Demko} 
on the inverses of band matrices and have been used, for 
example, in the proof of the $H^1$-stability of the 
$L_2$-orthogonal projection onto finite element spaces 
\cite{Bank-Yserentant}. Our proof works on a more abstract 
level than that in \cite{Malqvist-Peterseim} and is less 
centered around the single finite element basis functions. 
Implicitly, it compares their projections onto the mentioned 
space of rapidly oscillating functions with their iteratively 
calculated approximations and shows in this way that these 
decay exponentially with the distance to the assigned nodes.
It rests, as that of all related results, upon a local version 
of the Poincar\'{e} inequality and thus depends on the local 
contrast ratio. A different class of methods not suffering 
from such restrictions and condensing, similar to numerical 
homogenization methods, the features of the problem under 
consideration in a relatively low dimensional matrix can 
possibly be based on hierarchical matrices \cite{Hackbusch};
see the recent work of Bebendorf \cite{Bebendorf} in 
conjunction with the work of Hackbusch and Drechsler
\cite{Hackbusch-Drechsler}.

The present paper builds a bridge between numerical 
homogenization methods and multigrid-like iterative solvers. 
As a matter of fact, many numerical upscaling techniques 
bear close and often overlooked resemblance to techniques 
used by the multigrid community to develop fast iterative 
methods suitable for problems with rough coefficient 
functions. A prominent example of this kind, not unsimilar 
to our approach, is the technique used by Xu and Zikatanov 
\cite{Xu-Zikatanov} to construct coarse-level basis functions 
for algebraic multigrid methods, another the reference 
\cite{Kornhuber-Yserentant}, in which the direct application 
of a two-grid iterative method to numerical homogenization 
is studied and compared to methods like ours.

%%%%%%%%%%%%%%%%%%%%%%%%%%%%%%%%%%%%%%%%%%%%%%%%%%%%%%%%%%%%%%%%%%%%%%%%%%

%Section 2
\section{The equation and the basic approximation of its solutions}
\label{sec2}

The model problem considered in this paper is a standard second 
order differential equation in weak form with homogeneous 
Dirichlet boundary conditions on a polygonal domain $\Omega$ 
in $d=2$ or $3$ space dimensions. Its solution space is the 
Sobolev space $H_0^1(\Omega)$ and the associated bilinear 
form reads
\begin{equation}    \label{eq2.1}
a(u,v)=\int_\Omega \nabla u\cdot A\nabla v\dx.
\end{equation}
The matrix $A$ is a function of the spatial variable $x$
with measurable entries and is assumed to be symmetric 
positive definite. We assume for simplicity that 
\begin{equation}    \label{eq2.2}
\delta\sp|\eta|^2\leq\,\eta\cdot A(x)\eta\,\leq M|\eta|^2
\end{equation}
holds for all $\eta\in\mathbb{R}^d$ and almost all 
$x\in\Omega$, where $|\eta|$ denotes the euclidean norm 
of $\eta$ and $\delta$ and $M$ are positive constants.
In the same way as in \cite{Kornhuber-Yserentant} it is 
possible to replace this condition by an, at least in the
quantitative sense, weaker local condition on the contrast 
ratio. The condition (\ref{eq2.2}) guarantees that the 
bilinear form (\ref{eq2.1}) is an inner product on 
$H_0^1(\Omega)$. It induces the energy norm $\|\cdot\|$, 
which is equivalent to the original norm on this space. 
The boundary value problem
\begin{equation}    \label{eq2.3}
a(u,v)=f^*(v),\quad v\in H_0^1(\Omega),
\end{equation}
possesses by the Lax-Milgram theorem under the condition 
(\ref{eq2.2}) for all bounded linear functionals $f^*$ 
on $H_0^1(\Omega)$ a unique solution $u\in H_0^1(\Omega)$. 

We cover the domain $\Omega$ with a triangulation $\T$,
consisting of triangles in two and of tetrahedrons in 
three space dimensions. We assume that the elements in 
$\T$ are shape regular but do not require that $\T$ is 
quasiuniform. Associated with $\T$ is the conforming, 
piecewise linear finite element subspace $\S$ of 
$H_0^1(\Omega)$. A key ingredient of the methods 
discussed here is a bounded local linear projection 
operator 
\begin{equation}    \label{eq2.4}
\Pi:H_0^1(\Omega)\,\to\,\S:u\,\to\,\Pi u
\end{equation}
like that defined as follows. At first, the given function
$u\in H_0^1(\Omega)$ is locally, on the single elements
$t\in\T$, approximated by its $L_2$-orthogonal projection
onto the space of linear functions, regardless of the 
continuity across the boundaries of the elements.
In a second step, the values of these approximants at 
a vertex in the interior of the domain are replaced by 
a weighted mean, according to the contribution of the 
involved elements to the area or the volume of their 
union. The values at the vertices on the boundary are 
set to zero. Together, these values fix the projection 
$\Pi u$ of $u$ onto $\S$. For functions 
$u\in H_0^1(\Omega)$ then the estimates
\begin{equation}    \label{eq2.5}
|\Pi u|_1\leq c_1|\sp u\sp |_1, \quad 
\|h^{-1}(u-\Pi u)\|_0\leq c_2|\sp u\sp |_1
\end{equation}
hold, where $\|\cdot\|_0$ denotes the $L_2$-norm,
$|\cdot|_1$ the $H^1$-seminorm, and $h$ is an
elementwise constant function whose value on the
interior of a given element is its diameter. The 
first condition means that the projection operator 
(\ref{eq2.4}) is stable with respect to the 
$H^1$-norm and therefore also with respect to the 
energy norm. The second condition is an approximation 
property. The constants $c_1$ and $c_2$ depend, as 
with any other reasonable choice of $\Pi$, only on 
the shape regularity of the finite elements, but 
not on their size. Quasi-interpolation operators 
are a common tool in finite element theory. The 
use of quasi-interpolation operators that are at 
the same time projections onto the finite element 
space under consideration can be traced back 
to the work of Brenner \cite{Brenner} and Oswald 
\cite{Oswald}. The operator described above falls 
into this category and is analyzed in the appendix 
to this paper. A comprehensive recent presentation 
of such constructions can be found in 
\cite{Ern-Guermond}.

The kernel $\V=\ker\Pi$ of $\Pi$ is a closed subspace of 
$H_0^1(\Omega)$ and therefore itself a Hilbert space.
We can therefore introduce the $a$-orthogonal projection 
operator~$C$ from~$H_0^1(\Omega)$ onto the kernel of 
$\Pi$ and moreover the finite dimensional subspace
\begin{equation}    \label{eq2.6}
\W=\{v-Cv\,|\,v\in\sp \S\}
\end{equation}
of the $a$-orthogonal complement of the kernel of $\Pi$.
The dimension of $\W$ and of the finite element space 
$\S$ coincide as $v\in\S$ can be recovered from 
$v-Cv$ via
\begin{equation}    \label{eq2.7}
v\,=\,\Pi\sp (v-Cv).
\end{equation}
Analogously to the approach in \cite{Malqvist-Peterseim}, 
we discretize the equation (\ref{eq2.3}) using $\W$ 
both as trial and test space. The following representation
of the approximate solution is based on observations made
in \cite{Peterseim}, where a computationally more 
advantageous nonsymmetric variant with $\S$ as trial 
and $\W$ as test space is advocated.

\begin{lemma}[Peterseim \cite{Peterseim}, cf. also 
M\r{a}lqvist and Peterseim \cite{Malqvist-Peterseim}]
\label{lm2.1}
The equation 
\begin{equation}    \label{eq2.8}
a(w,\chi)=f^*(\chi), \quad \chi\in\W,
\end{equation}
possesses a unique solution $w\in\W$, namely
the projection
\begin{equation}    \label{eq2.9}
w=\Pi u-C\Pi u
\end{equation}
of the exact solution $u\in H_0^1(\Omega)$ 
of equation {\rm (\ref{eq2.3})}
onto the space $\W$. The error
\begin{equation}    \label{eq2.10}
u-w=Cu
\end{equation}
is the $a$-orthogonal projection $Cu$ of 
the solution $u$ onto the kernel of $\Pi$.
\end{lemma}

\begin{proof}
As $\Pi$ is a projection operator, $u-\Pi u$ is 
contained in the kernel of $\Pi$. The difference
of the exact solution $u$ and the function 
(\ref{eq2.9}) can therefore be written as
\begin{displaymath}    
u-w=(u-\Pi u)-C(u-\Pi u)+Cu=Cu.
\end{displaymath}
Because the functions in $\W$ are $a$-orthogonal
to the functions in the range of $C$,
\begin{displaymath}    
a(u-w,\chi)=\sp 0, \quad \chi\in\W, 
\end{displaymath}    
follows. That is, this $w$ is the unique solution 
of the equation (\ref{eq2.8}).
\end{proof}

The energy norm of $Cu$ can easily be estimated 
for a right-hand side 
\begin{equation}    \label{eq2.11}
f^*(v)=\int_\Omega fv\dx.
\end{equation}
One obtains in this way the following, rather 
surprising error estimate.

\begin{theorem}
[M\r{a}lqvist and Peterseim \cite{Malqvist-Peterseim}]  
\label{thm2.2}
For right-hand sides of the form {\rm (\ref{eq2.11})},
the approximate solution {\rm (\ref{eq2.9})} 
satisfies the energy norm error estimate
\begin{equation}    \label{eq2.12}
\|u-w\|\leq c\,\|hf\|_0,
\end{equation}
where the constant $c$ depends only on the constants 
$c_2$ from {\rm (\ref{eq2.5})} and $\delta$ from 
{\rm (\ref{eq2.2})}.
\end{theorem}

\begin{proof}
The proof is based on the representation (\ref{eq2.10}) 
of the error as $a$-orthogonal projection $Cu$ of 
the solution $u$ onto the kernel of $\Pi$ and starts 
from the identity
\begin{displaymath}
\|Cu\|^2=\,a(u,Cu-\Pi\sp Cu)=(f,Cu-\Pi\sp Cu),
\end{displaymath}
from which one obtains the estimate
\begin{displaymath}
\|Cu\|^2\leq\,\|hf\|_0\|h^{-1}(Cu-\Pi\sp Cu)\|_0.
\end{displaymath}
The second factor on the right-hand side is estimated
with help of the error estimate from (\ref{eq2.5}),
and the $H^1$-seminorm of $Cu$, then with (\ref{eq2.2}) 
by its energy norm.
\end{proof}

Remarkably, neither the smoothness of the solution $u$ 
nor the regularity properties of the equation enter 
into this error estimate. The size of the error bound 
is determined by the local behavior of the right-hand
side $f$.

%%%%%%%%%%%%%%%%%%%%%%%%%%%%%%%%%%%%%%%%%%%%%%%%%%%%%%%%%%%%%%%%%%%%%%%%%%

%Section 3
\section{Localization}
\label{sec3}

Let $x_1,x_2,\ldots,x_n$ be the vertices of the 
elements in the triangulation $\T$ and 
let~$\varphi_1,\varphi_2,\ldots,\varphi_n$ be the 
piecewise linear hat functions assigned to these
nodes. The $\varphi_i$ assigned to the nodes in 
the interior of the domain $\Omega$ then form a 
basis of the finite element space $\S$, and the 
corresponding functions $\varphi_i-C\varphi_i$ a 
basis of the trial space (\ref{eq2.6}). It has 
been shown in \cite{Malqvist-Peterseim} that 
these basis functions decay exponentially with 
the distance to the assigned nodes and can 
therefore be replaced by localized counterparts. 
We deviate here from the arguments there and 
utilize the theory of iterative methods to prove 
a result of similar kind. Let $\omega_i$ be the 
union of the finite elements with vertex $x_i$ 
and let 
\begin{equation}    \label{eq3.1} 
\V_i=\{v-\Pi v\,|\,v\in H_0^1(\omega_i)\}.
\end{equation}
The functions in $\V_i$ vanish outside a small 
neighborhood of the vertex $x_i$, depending on 
the choice of $\Pi$. For the exemplary operator 
mentioned earlier, this neighborhood consists of 
the two shells of elements surrounding $x_i$.
The $\V_i$ are closed subspaces of the kernel $\V$ 
of $\Pi$. This can be seen as follows. Let the 
$v_k\in\V_i$ converge to the function $v\in\V$. 
As the $v_k$ are piecewise linear outside 
$\omega_i$, the same holds for $v$. Thus there 
exists a function $\varphi$ in the finite element 
space $\S$ such that $v-\varphi\in H_0^1(\omega_i)$. 
Because $\Pi v=0$ and $\Pi\varphi=\varphi$, then
\begin{equation}    \label{eq3.2}
v=(v-\varphi)-\Pi(v-\varphi)\in\V_i.
\end{equation}
Let $P_i$ be the $a$-orthogonal projection from 
$H_0^1(\Omega)$ to $\V_i$, defined via the 
equation
\begin{equation}    \label{eq3.3}
a(P_iv,v_i)=a(v,v_i),\quad v_i\in\V_i.
\end{equation}
Introducing the, with respect to the bilinear 
form (\ref{eq2.1}), symmetric operator
\begin{equation}    \label{eq3.4}
T=P_1+P_2+\cdots+P_n,
\end{equation}
the approximation spaces replacing $\W$ are built 
up with the help of the bounded linear operators 
$F_\nu$ from $H_0^1(\Omega)$ to $\V$ that are, 
starting from $F_0u=0$, defined via
\begin{equation}    \label{eq3.5}
F_{\nu+1}u=F_\nu u+T(u-F_\nu u). 
\end{equation}
The correction $T(u-F_\nu u)$ is the sum of its
components $d_i=P_i(u-F_\nu u)$ in the subspaces 
$\V_i$ of $\V$, the solutions $d_i\in\V_i$ of 
the local equations
\begin{equation}    \label{eq3.6}
a(d_i,v_i)=a(u,v_i)-a(F_\nu u,v_i), \quad v_i\in\V_i.
\end{equation}
The new trial and test spaces are the spaces 
$\W_\ell$ spanned by the functions 
\begin{equation}    \label{eq3.7}
\varphi_i-F_\nu\varphi_i, \quad \nu=0,1,\ldots,\ell, 
\end{equation}
attached to the nodes $x_i$ in the interior 
of the domain $\Omega$. In contrast to their 
counterparts $\varphi_i-C\varphi_i$ spanning 
the original space $\W$ they have a local 
support, which expands layer by layer with 
the number $\nu$ of iterations.

To study the approximation properties of these 
spaces $\W_\ell$, we consider optimally or almost 
optimally chosen fixed linear combinations
\begin{equation}    \label{eq3.8}
C_\ell=\sum_{\nu=0}^\ell\alpha_{\ell\nu}F_\nu,
\quad 
\sum_{\nu=0}^\ell\alpha_{\ell\nu}=1,
\end{equation}
of the operators $F_\nu$ as approximations of the 
$a$-orthogonal projection $C$. These operators 
$C_\ell$ serve here solely as a tool and do not 
need to be explicitly accessible. Our analysis 
is based on the theory of additive Schwarz or 
subspace decomposition methods, here applied to
an equation in the kernel $\V$. Key is the 
following lemma.

\begin{lemma}       \label{lm3.1}
For all $v\in\V$, there is a with respect 
to the energy norm stable decomposition 
$v=v_1+\cdots +v_n$ of $v$ into functions 
$v_i$ in the spaces $\V_i$, such that
\begin{equation}    \label{eq3.9}
\sum_{i=1}^n\|v_i\|^2\leq K_1\|v\|^2
\end{equation}
holds, where the constant $K_1$ depends only 
on the constants $c_1$ and $c_2$ from 
{\rm (\ref{eq2.5})}, on the shape regularity 
of the finite elements, and on the contrast 
ratio $M/\delta$. Moreover, there is a 
constant $K_2$ such that
\begin{equation}    \label{eq3.10}
\|v\|^2\leq K_2\sum_{i=1}^n\|v_i\|^2
\end{equation}
holds for all such decompositions of $v$ 
into functions $v_i$ in the subspaces 
$\V_i$ of the kernel. This constant 
depends only on the shape regularity of 
the finite elements.
\end{lemma}

\begin{proof}
The upper estimate (\ref{eq3.10}) is rather trivial
because $K_2$ can be bounded in terms of the maximum 
number of the parts $v_i$ that do not vanish on a 
given element. We use that the $\varphi_i$ form a
partition of unity and prove that (\ref{eq3.9})
holds for the decomposition of a function $v$
in the kernel $\V$ of $\Pi$ into the parts
\begin{displaymath}
v_i=\varphi_i v-\Pi\sp (\varphi_i v), 
\quad i=1,\ldots,n,
\end{displaymath}
in $\V_i$. It suffices to prove that this decomposition 
is $H^1$-stable. By the $H^1$-stability of the projection 
$\Pi$ and the shape regularity of the finite elements 
one obtains
\begin{displaymath}
\sum_{i=1}^n|\sp v_i\sp |_1^2
\,\lesssim\,
\sum_{i=1}^n|\varphi_i v|_1^2
\,\lesssim\, 
|\sp v\sp |_1^2\sp +\|h^{-1}v\|_0^2.
\end{displaymath}
Using once more that $\Pi v=0$, the second term
on the right-hand side can, by the approximation 
property from (\ref{eq2.5}), be estimated as
\begin{displaymath}
\|h^{-1}v\|_0=\sp \|h^{-1}(v-\Pi v)\|_0
\lesssim\sp |\sp v\sp |_1.
\end{displaymath}
Since the $H^1$-seminorm can, because of the assumption 
(\ref{eq2.2}), be estimated by the energy norm and vice 
versa, the estimate (\ref{eq3.9}) follows and the proof 
is complete.
\end{proof}

With the help of the estimates (\ref{eq3.9}) and 
(\ref{eq3.10}) one can show that
\begin{equation}    \label{eq3.11}
1/K_1\sp a(v,v)\leq a(Tv,v)\leq K_2\sp a(v,v)
\end{equation}
holds for the functions $v$ in the kernel $\V$ 
of $\Pi$. The spectrum of $T$, seen as an 
operator from $\V$ to itself, is therefore a 
compact subset of the interval with the endpoints 
$1/K_1$ and $K_2$. Because
$I-F_\nu=(I-T)^\nu$ and $F_\nu C=F_\nu$,
\begin{equation}    \label{eq3.12}
C-C_\ell=
\bigg\{\sum_{\nu=0}^\ell\alpha_{\ell\nu}(I-T)^\nu\bigg\}\,C.
\end{equation}
Using the spectral mapping theorem and the fact
that the norm of a bounded, symmetric operator 
from a Hilbert space to itself is equal to its
spectral radius, one gets therefore similarly 
to the finite dimensional case the following 
error estimate.

\begin{lemma}       \label{lm3.2}
If the weights $\alpha_{\ell\nu}$ are optimally 
chosen, the estimate
\begin{equation}    \label{eq3.13}
\|C u-C_\ell u\|\leq 
\frac{2\sp q^{\sp \ell}}{1+q^{\sp 2\ell}}
\,\|Cu\|
\end{equation}
holds for all $u\in H_0^1(\Omega)$, where 
the convergence rate 
\begin{equation}    \label{eq3.14}
q=\frac{\sqrt{\kappa}-1}{\sqrt{\kappa}+1}
\end{equation}
is determined by the condition number 
$\kappa\leq K_1K_2$ of the operator 
{\rm (\ref{eq3.4})} seen as bounded, 
symmetric operator from the subspace 
$\V$ of $H_0^1(\Omega)$ to itself.
\end{lemma}

The distance of $C_\ell u$ to $C u$ can thus be 
estimated in terms of the two constants $K_1$ and 
$K_2$ and the norm of $Cu$ and tends exponentially 
in the number $\ell$ of iterations to zero. The 
lemma is basically a reformulation of a standard 
result from the theory of subspace decomposition 
methods \cite{Xu}, \cite{Yserentant}, where $Cu$ is 
here the solution of the equation and the $C_\ell u$ 
are its iteratively generated approximations. We 
refer to \cite{Kornhuber-Yserentant} for the 
missing details of the here only sketched proof. 

\begin{theorem}     \label{thm3.3}
Let $w$ and $w_\ell$ be the best approximations 
of the solution $u$ of the original equation 
{\rm (\ref{eq2.3})} in $\W$ and $\W_\ell$ with 
respect to the energy norm. Then
\begin{equation}    \label{eq3.15}
\|u-w_\ell\|\,\leq\,
\bigg(1+\frac{2\sp q^{\sp \ell}}{1+q^{\sp 2\ell}}\bigg)\|u-w\|+
\frac{2\sp q^{\sp \ell}}{1+q^{\sp 2\ell}}\,\|u-\Pi u\|.
\end{equation}
\end{theorem}

\begin{proof}
Because $w_\ell$ is the best approximation of $u$
by a function in $\W_\ell$, because the function 
$\Pi u-C_\ell\Pi u$ is contained in this space,
and because $w=\Pi u-C\Pi u$, we~have
\begin{displaymath}
\|u-w_\ell\|\leq\|u-(\Pi u-C_\ell\Pi u)\|
=\|(u-w)-(C\Pi u-C_\ell\Pi u)\|.
\end{displaymath}
The distance of $C\Pi u$ and $C_\ell\Pi u$ can 
according to Lemma~\ref{lm3.2} be estimated by 
the energy norm of $C\Pi u$. As 
$C\Pi u=(u-w)-(u-\Pi u)$, this leads to 
(\ref{eq3.15}).
\end{proof}

We conclude that logarithmically many iteration steps 
$\nu$ or even less, depending on the behavior of the 
energy norm of $u-\Pi u$, suffice to reach the same 
level of accuracy as with the original space $\W$ 
based on the exact projection $C$.
The best approximation $w_\ell$ of the solution $u$
in the space $\W_\ell$ can itself again be calculated
iteratively by an additive or multiplicative Schwarz 
or subspace decomposition method based on the splitting 
of $\W_\ell$ into the finite element space $\S$, that 
provides for the global exchange of information, and 
local subspaces like
\begin{equation}    \label{eq3.16}
\W_{\ell,i}=\,
\span\{\sp\varphi_i-F_\nu\varphi_i\,|\,\nu=0,\ldots,\ell\,\},
\end{equation}
bearing the information on the fine scale structure
of the solution. 

The proof of Theorem~\ref{thm3.3} shows that one can 
replace the $\ell+1$ functions (\ref{eq3.7}) attached 
to the nodes $x_i$ in the interior of $\Omega$ by a 
single, fixed linear combination
\begin{equation}    \label{eq3.17}
\sum_{\nu=0}^\ell\alpha_{\ell\nu}(\varphi_i-F_\nu\varphi_i),
\quad 
\sum_{\nu=0}^\ell\alpha_{\ell\nu}=1,
\end{equation}
without sacrificing the error bound (\ref{eq3.15}). 
This considerably reduces the size of the spaces 
$\W_\ell$. The problem is that the optimum  
coefficients $\alpha_{\ell\nu}$ depend on the 
entire spectrum of the operator $T$ from the 
kernel $\V$ to itself. In a much simplified 
variant, approximations $C_\ell$ of the projection 
$C$ are determined via a recursion
\begin{equation}    \label{eq3.18}
C_{\ell+1}u=C_\ell u+\omega\sp T(u-C_\ell u),
\end{equation}
where $C_0 u=0$ is set and $\omega\geq 1/K_2$ is 
a damping parameter whose optimal value depends 
again on the end points of the spectrum. Because 
in this case
\begin{equation}    \label{eq3.19}
C-C_\ell=(I-\omega\sp T)^\ell C,
\end{equation}
convergence is guaranteed at least for $\omega<2/K_2$, 
and for $\omega=1/K_2$ in particular. The functions 
$\varphi_i-C_\ell\varphi_i$ spanning the new spaces 
$\W_\ell$ can be calculated in the same way as the 
functions $\varphi_i-F_\nu\varphi_i$. The convergence 
rate degrades, however, with this version from 
(\ref{eq3.14}) to a value not better than
\begin{equation}    \label{eq3.20}
q=\frac{\kappa-1}{\kappa+1},
\end{equation}
and in the extreme case to $q=1-1/(K_1K_2)$ if 
$\omega=1/K_2$ is chosen. This means that 
possibly a larger number $\ell$ of iterations 
and layers, respectively, is needed.

%%%%%%%%%%%%%%%%%%%%%%%%%%%%%%%%%%%%%%%%%%%%%%%%%%%%%%%%%%%%%%%%%%%%%%%%%%

%Section 4
\section{Discretization}
\label{sec4}

The infinite dimensional subspaces $\V_i$ of the kernel
of the projection $\Pi$ have to be replaced by discrete 
counterparts to obtain a computationally feasible method. 
We start from a potentially very strong, uniform or 
nonuniform refinement $\T'$ of the triangulation $\T$, 
bridging the scales and resolving the oscillations of 
the coefficient functions, and a finite element space 
$\S'\subseteq H_0^1(\Omega)$ that consists of the 
continuous functions whose restrictions to the elements 
in $\T'$ are linear. The whole theory then literally 
transfers to the present situation replacing only 
the continuous solution space $H_0^1(\Omega)$ and its
subspaces by their discrete counterparts. The only
modification concerns the construction of the stable
decomposition of the functions~$v$ in the kernel of
$\Pi$ into a sum of functions in the corresponding 
local subspaces
\begin{equation}    \label{eq4.1} 
\V_i=\{v-\Pi v\,|\,v\in \S'\cap H_0^1(\omega_i)\}.
\end{equation}
To construct such a decomposition, we use the interpolation 
operator $\I:C(\bar{\Omega})\to\S'$ that interpolates 
at the nodes of usual kind and reproduces the functions 
in $\S'$. As the operator $\I$ is linear and the 
$\varphi_i$ form a partition of unity, we can decompose 
the functions $v\in\S'$ in the kernel of $\Pi$ into 
the sum of the functions 
\begin{equation}    \label{eq4.2}
v_i=\I(\varphi_i v)-\Pi\sp (\I(\varphi_i v))
\end{equation}
in the modified subspaces $\V_i$. The stability of 
this decomposition in the sense of (\ref{eq3.9}) 
can be deduced in the same way as the stability of 
the decomposition in the proof of Lemma~\ref{lm3.1} 
since for the functions $v\in\S'$ in the kernel of 
$\Pi$ an estimate
\begin{equation}    \label{eq4.3}
|\I(\varphi_i v)|_1\leq c\,|\varphi_i v|_1
\end{equation}
holds, which is shown separately for the single
elements $t\in\T'$ using that the restrictions of 
such functions $\varphi_i v$ to these elements are 
second order polynomials.

%%%%%%%%%%%%%%%%%%%%%%%%%%%%%%%%%%%%%%%%%%%%%%%%%%%%%%%%%%%%%%%%%%%%%%%%%

%Appendix
\section*{Appendix. Analysis of a local projection operator}

For the convenience of the reader, we give in this appendix 
a comparatively detailed proof of the estimates (\ref{eq2.5}) 
for the exemplary local linear projection operator~$\Pi$ 
described in Section~\ref{sec2}. Let $\N$ be the set of the 
indices of the vertices $x_i$ of the finite elements in 
the interior of $\Omega$. The operator $\Pi$ can then be 
written as
\begin{displaymath}
\Pi v=\sum_{i\,\in\N}\alpha_i\varphi_i,
\end{displaymath} 
with coefficients $\alpha_i$ that depend linearly on $v$ 
and are calculated as follows. At first, the given function 
$v$ is locally, separately on each single finite element 
and regardless of the continuity across the boundaries 
of the elements, approximated by its $L_2$-orthogonal 
projection onto the space of linear functions. The 
coefficient $\alpha_i$ is then a weighted mean of the 
values of these linear functions at the node $x_i$ under 
consideration, weighted according to the contribution of 
the involved elements to the area or the volume of the 
patch $\omega_i$, the support of the basis function 
$\varphi_i$. The constant functions $x\to\alpha_i$ 
satisfy the local $L_2$-norm estimate
\begin{displaymath}
\|\alpha_i\|_{0,\omega_i}\lesssim\|v\|_{0,\omega_i}
\end{displaymath}
over these patches, where the constant on the 
right-hand side depends solely on the space dimension.
Let the patch $\omega_i$ be the union of the elements 
$t_1,\ldots,t_n$ and let $v_1,\ldots,v_n$ be the 
corresponding local $L_2$-orthogonal projections of $v$
onto the space of linear functions. Let $|\omega_i|$ 
and $|t_k|$ be the areas or volumes of $\omega_i$ and 
the $t_k$. Then
\begin{displaymath}
\|\alpha_i\|_{0,\omega_i}^2=\,
|\omega_i|\,\bigg(\frac{1}{|\omega_i|}\sum_{k=1}^n|t_k|v_k(x_i)\bigg)^2
\leq\,\sum_{k=1}^n|t_k|v_k(x_i)^2.
\end{displaymath}
Transformation to a reference element yields the estimate
\begin{displaymath}
\sum_{k=1}^n|t_k|v_k(x_i)^2\lesssim\sum_{k=1}^n\|v_k\|_{0,t_k}^2
\end{displaymath}
of the right-hand side, with a constant that depends 
only on the space dimension. As the $L_2$-norms of 
the linear functions $v_k$ over the $t_k$ are less 
than or equal to the $L_2$-norms of $v$ over the $t_k$, 
this proves the estimate above for the $\alpha_i$.

To prove the estimates (\ref{eq2.5}), we use that the
functions in $H_0^1(\Omega)$ can be considered as
functions in $H^1(\mathbb{R}^2)$ and $H^1(\mathbb{R}^3)$, 
respectively, with value zero outside $\Omega$. Because 
the hat functions $\varphi_i$ form a partition of unity 
on $\Omega$, 
\begin{displaymath}
v-\Pi v=\sum_{i\,\in\N}\varphi_i\sp(v-\alpha_i)
+\sum_{i\,\notin\N}\varphi_i v.
\end{displaymath}
On a given element only the functions $\varphi_i$ 
assigned to its vertices are different from zero. 
The square of the $H^1$-seminorm of the error can 
therefore be estimated as
\begin{displaymath}
|v-\Pi v|_1^2\,\lesssim
\sum_{i\,\in\N}|\varphi_i\sp (v-\alpha_i)|_1^2
+\sum_{i\,\notin\N}|\varphi_i v|_1^2.
\end{displaymath}
Let $B_i$ be the ball with center $x_i$ of minimum
diameter that covers the patch $\omega_i$ and let
$h_i$ be its radius. As 
$|\nabla\varphi_i|\lesssim h_i^{-1}$ by the shape 
regularity of the elements, 
\begin{displaymath}
|\varphi_i\sp(v-\alpha_i)|_1\lesssim
h_i^{-1}\|v-\alpha_i\|_{0,\omega_i}+
|\sp v\sp|_{1,\omega_i}.
\end{displaymath}
As the linear functional $v\to\alpha_i$ reproduces 
the value of constant functions,
\begin{displaymath}
\|v-\alpha_i\|_{0,\omega_i}\lesssim
\|v-\alpha\|_{0,\omega_i}\leq\|v-\alpha\|_{0,B_i}
\end{displaymath}
holds for all constants $\alpha$, and, in particular,
for the mean value $\alpha$ of the function $v$ over 
the ball $B_i$. The Poincar\'{e} inequality for 
balls leads therefore to the estimate
\begin{displaymath}
\|v-\alpha_i\|_{0,\omega_i}\lesssim h_i|\sp v\sp|_{1,B_i},
\end{displaymath}
with a constant that is, of course, independent 
of the radius $h_i$ of $B_i$. For the terms 
associated with the inner vertices, thus finally 
\begin{displaymath}
|\varphi_i\sp (v-\alpha_i)|_1\lesssim|\sp v\sp|_{1,B_i}.
\end{displaymath}
The boundary terms can be treated with a local variant
of the Friedrichs inequality. Here we prefer to proceed
in a similar way as with the inner vertices. At first,
\begin{displaymath}
|\varphi_i v|_1\lesssim
h_i^{-1}\|v\|_{0,\omega_i}+|\sp v\sp|_{1,\omega_i}.
\end{displaymath}
As polygonal domain, $\Omega$ satisfies an exterior 
cone condition. That is, for the $B_i$ assigned to 
the vertices $x_i$ on the boundary of $\Omega$ 
there is constant $c$ such that
\begin{displaymath}    
|B_i\cap\Omega|\leq c\,|B_i\!\setminus\!\Omega\sp|.
\end{displaymath}
This constant is independent of $i$ and depends at 
most on an upper bound for the diameters of the balls. 
The $L_2$-distance of a function $v$ in $L_2(B_i)$ 
to its mean value over the part of $B_i$ outside 
$\Omega$ can therefore, analogously to the reasoning 
above, be estimated by its $L_2$-distance to the mean 
value over $B_i$ itself. As the mean value of the 
given functions $v$ over the part of $B_i$ outside 
of $\Omega$ is zero, this leads, again by means of 
the Poincar\'{e} inequality for balls, to the 
estimate
\begin{displaymath}
\|v\|_{0,\omega_i}\lesssim h_i|\sp v\sp|_{1,B_i}.
\end{displaymath}
For the terms associated with the vertices on the 
boundary, then
\begin{displaymath}
|\varphi_i v|_1\lesssim |\sp v\sp|_{1,B_i}. 
\end{displaymath}
Since the balls $B_i$ form, because of the shape regularity 
of the finite elements, a locally finite covering of $\Omega$, 
one obtains finally the estimate
\begin{displaymath}
|v-\Pi v|_1\lesssim |\sp v\sp|_1,
\end{displaymath}
with a constant that depends only on the shape regularity 
of the finite elements and the constant from the exterior 
cone condition in the form given above. This implies the 
stability of $\Pi$. Using that for all square integrable 
functions $w$
\begin{displaymath}
\|h^{-1}\varphi_i w\|_0\lesssim h_i^{-1}\|w\|_{0,\omega_i}
\end{displaymath}
holds, the approximation property follows by the same 
arguments.

%%%%%%%%%%%%%%%%%%%%%%%%%%%%%%%%%%%%%%%%%%%%%%%%%%%%%%%%%%%%%%%%%%%%%%%%%

\bibliographystyle{amsplain}
\bibliography{paper}

%%%%%%%%%%%%%%%%%%%%%%%%%%%%%%%%%%%%%%%%%%%%%%%%%%%%%%%%%%%%%%%%%%%%%%%%%

\end{document}